\documentclass[12pt]{amsart}

\usepackage{amsmath,amssymb,epic,lscape}

\newtheorem{theorem}{Theorem}[section]

\newtheorem{lemma}[theorem]{Lemma}

\theoremstyle{definition}

\theoremstyle{remark}
\newtheorem{remark}[theorem]{Remark}

\numberwithin{equation}{section}

\def\DJ{{\hbox{D\kern-.8em\raise.15ex\hbox{--}\kern.35em}}}
\def\DJo{$\;$\kern-.4em
    \hbox{D\kern-.8em\raise.15ex\hbox{--}\kern.35em okovi\'c}}

\def\ve{{\varepsilon}}
\def\vf{{\varphi}}

\def\CT{{\mbox{\rm CT}}}
\def\bQ{{\mbox{\bf Q}}}
\def\bZ{{\mbox{\bf Z}}}
\def\bC{{\mbox{\bf C}}}

\def\GL{{\mbox{\rm GL}}}

\def\Ort{{\mbox{\rm O}}}

\renewcommand{\subjclassname}{\textup{2000} Mathematics Subject
Classification}

\begin{document}

\title[An octonion algebra]
{An octonion algebra originating in combinatorics}

\author[D.\v{Z}. \DJ okovi\'{c} and K. Zhao]
{Dragomir \v{Z}. \DJ okovi\'{c} and Kaiming Zhao}

\address{Department of Pure Mathematics, University of Waterloo,
Waterloo, Ontario, N2L 3G1, Canada}

\address{Department of Mathematics, Wilfrid Laurier University,
Waterloo, Ontario, N2L 3C5, Canada and Institute of Mathematics,
Academy of Mathematics and System Sciences, Chinese Academy of
Sciences, Beijing 100190, P.R. China}

\email{djokovic@uwaterloo.ca} \email{kzhao@wlu.ca}

\thanks{
The first author was supported by an NSERC Discovery Grant, and
the second by the NSERC and the NSF  of China (Grant 10871192).}


\date{}

\begin{abstract}
C.H. Yang discovered a polynomial version of the classical Lagrange
identity expressing the product of two sums of four squares
as another sum of four squares. He used it to give short proofs of
some important theorems on composition of $\delta$-codes
(now known as $T$-sequences). We investigate the possible new versions
of his polynomial Lagrange identity. Our main result shows that all
such identities are equivalent to each other.
\end{abstract}

\maketitle \subjclassname{ 17A75, 05B30 }

\section{Introduction}

C.H. Yang \cite{Y1,Y2,Y3} discovered a polynomial version (see below)
of the classical Lagrange identity on the product of two sums of four
squares. He used his identity to give short and elegant proofs of
some important theorems \cite{Y3} in combinatorics of binary and
ternary sequences. These results provide new methods for the
construction of several classes of combinatorial objects such as
$T$-sequences, orthogonal designs, and Hadamard matrices \cite{SY,Y3}.
Our motivation and the main goal was to investigate the possible new
versions of the polynomial Lagrange identity. However, our main
result shows that all such identities are equivalent to
each other (for the precise statement see Theorem \ref{glavna}).

Let $A=\bZ[z, z^{-1}]$ be the Laurent polynomial ring over the
integers $\bZ$. For any $f=f(z)=\sum_k a_k z^k\in A$, $a_k\in\bZ$,
we define its {\it conjugate} as $f^*=f(z^{-1})$. We also say that $a_0$
is the {\em constant term} of $f$ and write $\CT(f)=a_0$.
Note that $\CT(ff^*)=\sum_k a_k^2$.
Let $A_0$ be the fixed subring of $A$ under the conjugation, i.e.,
$A_0=\{f\in A\,|\, f=f^*\}$. It is easy to see that $A_0=\bZ[z+z^{-1}]$.
We embed the polynomial ring $\bZ[t]$ into $A$ by sending $t\to z+z^{-1}$
and view $A$ as a $\bZ[t]$-algebra.
It is easy to check that $A=A_0\oplus A_0 z$.
Thus $A$ is a free $\bZ[t]$-module of rank 2.

We are mainly interested in the free $A$-module $E=A^4$. When viewed
as a $\bZ[t]$-module it is again free but now its rank is 8.
For any $x=(x_0,x_1,x_2,x_3)\in E$ we define its {\it norm } as
$N(x)=\sum_k x_k x_k^*.$ Thus $N:E\to A_0$ is a quadratic form on the
$A_0$-module $E$.

Now we state the Lagrange identity for Laurent polynomials
(modified Theorem 1 in \cite{Y3}). Let $x=(x_k),\, y=(y_k)\in E$ and
define $(p,q,r,s)\in E$ by {\em Yang's formulae}:
\begin{eqnarray*}
p &=& x_0y_0-x_1y_1^*-x_2y_2^*-x_3y_3^*, \\
q &=& x_0y_1+x_1y_0^*+x_2^*y_3^*-x_3^*y_2^*, \\
r &=& x_0y_2-x_1^*y_3^*+x_2y_0^*+x_3^*y_1^*, \\
s &=& x_0y_3+x_1^*y_2^*-x_2^*y_1^*+x_3y_0^*.
\end{eqnarray*}
We can use these formulae to define an $A_0$-bilinear multiplication
``$\circ$'' on $E$ by $x\circ y=(p,q,r,s)$.
It is straightforward to verify that the Lagrange identity
$N(x\circ y)=N(x)N(y)$ is indeed valid.

We shall see in the next section that $(E,\circ)$ is in fact an
octonion algebra over $\bZ[t]$. In Section \ref{OrtGr} we give an
explicit description of the orthogonal group $\Ort(N)$ of the
pair $(E,N)$, see Theorem \ref{Ort-2}.
In Section \ref{proizvodi}, we show that all $A_0$-bilinear
multiplications on $E$ satisfying the Lagrange identity are
equivalent, in the sense defined there, to the above
multiplication ``$\circ$''.

Our main result can probably be generalized by replacing $\bZ$
with a more general commutative ring. We decided not to pursue
this here in order to preserve the essentially combinatorial
flavour of the original problem.

We are grateful to the referee for correcting a couple of errors
in the original proof of Lemma \ref{druga}, for giving the
stronger result in Theorem 4.3, and for his other detailed
comments and suggestions.

\section{Yang formulae define an octonion algebra} \label{Octonion}

We make $A=\bZ[z,z^{-1}]$ into a $\bZ[t]$-algebra via the homomorphism
$\bZ[t]\to A$ sending $t\to z+z^{-1}$. We shall often identify $\bZ[t]$
with its image $A_0$ under this homomorphism.
According to the definition in \cite[Chapter III, \S2]{NB}, $A$ is a
quadratic $\bZ[t]$-algebra. By using the basis $\{1,z\}$, we see
that its type is $(-1,t)$. Indeed we have $z^2=-1+tz$. Moreover,
$(A,\ast)$ is an example of a Cayley algebra (see {\em loc. cit.}).

Let $H=A\times A$, a free $A$-module of rank 2. We shall also view
it as a free $\bZ[t]$-module of rank 4 with basis
$(1,0)$, $(z,0)$, $(0,1)$, $(0,z)$. We make $H$ into an associative
noncommutative algebra by using the Cayley--Dickson process,
i.e., we define the multiplication in $H$ by
$$ (a,b)(c,d)=(ac-d^*b,bc^*+da), \quad a,b,c,d\in A. $$
The involution ``$\ast$'' on $A$ extends to an involutory
anti-automorphism of $H$ by setting
$$ (a,b)^* = (a^*,-b), \quad a,b\in A. $$
Thus $H$ is an example of a quaternion algebra over $\bZ[t]$, see
{\em loc. cit.} No. 5, Example 2. As such, it has type $(-1,t,-1)$.

Let $E=H\times H$, a free $A$-module of rank 4. We shall also view
it as a free $\bZ[t]$-module of rank 8 with basis
\begin{eqnarray*}
&& e_0=(1,0,0,0),\quad e'_0=ze_0=(z,0,0,0), \\
&& e_1=(0,1,0,0),\quad e'_1=ze_1=(0,z,0,0), \\
&& e_2=(0,0,1,0),\quad e'_2=ze_2=(0,0,z,0), \\
&& e_3=(0,0,0,1),\quad e'_3=ze_3=(0,0,0,z).
\end{eqnarray*}
We make $E$ into a nonassociative algebra by using
the Cayley--Dickson process once more. Thus, we define
the multiplication in $E$ by
$$ (u,v)(x,y)=(ux-y^*v,vx^*+yu), \quad u,v,x,y\in H. $$
The involution ``$\ast$'' on $H$ extends to one on $E$ by setting
$$ (u,v)^* = (u^*,-v), \quad u,v\in H. $$
Thus $E$ is an example of an octonion algebra over $\bZ[t]$,
see \cite[Chapter III, Appendix]{NB}. By using the above basis,
we find that the type of this octonion algebra is $(-1,t,-1,-1)$.

Let us write the above octonion multiplication in terms of the
$A$-basis $\{e_0,e_1,e_2,e_3\}$. For
$x=(x_0,x_1,x_2,x_3)$ and $y=(y_0,y_1,y_2,y_3)$, we find that
\begin{eqnarray*}
xy &=& (x_0y_0-x_1y_1^*-x_2y_2^*-x_3^*y_3, \\
&& \, x_0y_1+x_1y_0^*+x_2^*y_3-x_3y_2^*, \\
&& \, x_0y_2-x_1^*y_3+x_2y_0^*+x_3y_1^*, \\
&& \, x_0^*y_3+x_1y_2-x_2y_1+x_3y_0).
\end{eqnarray*}
Hence, the map sending $(x_0,x_1,x_2,x_3)\to(x_0,x_1,x_2,x_3^*)$
is an isomorphism of this octonion algebra with the algebra
$(E,\circ)$ defined by the Yang formulae.

\section{Orthogonal group} \label{OrtGr}

It is obvious that $U_1=\{x\in A\,\,|\,\,xx^*=1\}$ is the group of invertible elements of $A$. It consists of the elements
$\pm z^k, k\in \bZ$. In our proofs below we shall often use the
following obvious fact: The subset $U_1$ generates $A$ as an
additive group.

The ring homomorphisms $\vf:A\to\bC$ are parametrized by nonzero
complex numbers $w$: By definition, the homomorphism $\vf_w$
corresponding to $w$ sends the indeterminate $z$ to $w$. The
homomorphism $\vf_w$ is compatible with the involutions (conjugation
on $A$ and complex conjugation on $\bC$), i.e,
$\vf_w(x^*)=\overline{\vf_w(x)}$ for all $x\in A$, if and only if
$|w|=1$.

The following lemma will be used in the next section.

\begin{lemma} \label{jednacina}
If $m\in\bZ$ is not a square, then the equation $xx^*=m$ has no
solution in $A$.
\end{lemma}
\begin{proof}
Assume that $ff^*=m$ for some $f\in A$. By applying the homomorphism
$\vf_1$ we obtain $f(1)f^*(1)=m$. As $f^*(1)=f(1)\in\bZ$, we have
a contradiction.
\end{proof}

We also remark that, for $x\in A$, $\CT(xx^*)=0$ implies $x=0$,
and $\CT(xx^*)=1$ implies $xx^*=1$.

Denote the general linear group of the $A_0$-module $E$ by
$\GL(A_0,E)$. We introduce the orthogonal groups
\begin{eqnarray*}
\Ort(N) &=& \{\vf\in\GL(A_0,E)\,|\,
N(\vf(u))=N(u),\,\forall u\in E\},\\
\Ort_4(\bZ) &=& \{\vf\in\GL_4(\bZ)\,|\,
N(\vf(u))=N(u),\,\forall u\in\bZ^4\},
\end{eqnarray*}
where $\bZ^4$ is considered as an additive subgroup of $E=A^4$.

Let $Q:E\times E\to A_0$ be the polar form of $N$, i.e.,
\begin{eqnarray*}
Q(x,y) &=& N(x+y)-N(x)-N(y) \\
&=& \sum_{k=0}^3 \left( x_k^* y_k+ x_k y_k^* \right).
\end{eqnarray*}
Note that $Q$ is a symmetric $A_0$-bilinear $\Ort(N)$-invariant form,
and that $Q$ is nondegenerate, i.e., its kernel is 0.

Let us begin with a useful remark.

\begin{remark} \label{primedba}
An $A_0$-linear map $\vf:E\to E$ preserving $N$ is automatically bijective, and so $\vf\in\Ort(N)$. Indeed, as $Q$ is nondegenerate, the injectivity is obvious. If $S$ is the matrix of $Q$ with respect to
some $A_0$-basis of $E$, then $\det(\vf)^2 \det(S)=\det(S)\ne0$,
forcing $\det(\vf)=\pm1$, and $\vf$ is bijective.
\end{remark}

The following result is crucial  to this paper.

\begin{theorem} \label{crucial}
(a) The group $\Ort_4(\bZ)$ consists of entry
permutations on $\bZ^4$ with arbitrary sign changes.

(b) The set
$U_1'=\{x\in A\,\,|\,\,xx^*=-(z-z^{-1})^2\}$ is equal to
$(z-z^{-1})U_1$.

(c) The unit sphere $U_4=\{x\in E\,\,|\,\,N(x)=1\}$ is equal to
$\cup_{k=0}^3U_1e_k$.

(d) The sphere $U_4'=\{x\in E\,\,|\,\,N(x)=-(z-z^{-1})^2\}$ is equal to
$(z-z^{-1})U_4$.
\end{theorem}
\begin{proof}
(a) Fix $\alpha\in \Ort_4(\bZ)$. Let
$\alpha(e_i)=\sum_j a_{ij}e_j$ where $a_{ij}\in \bZ$. From
$1=N(\alpha(e_i))=\sum_j a_{ij}^2$ we see that, for
each $i$, exactly one of $a_{ij}$ is $\pm1$
and the other ones vanish. The rest follows easily.

(b) By applying the homomorphisms $\vf_{\pm1}$ to $xx^*=-(z-z^{-1})^2$,
we conclude that $x$ is divisible by $z-1$ and $z+1$, and so
we have $x=(z-z^{-1})y$ for some $y\in A$. As $xx^*=-(z-z^{-1})^2 yy^*$,
we deduce that $yy^*=1$, i.e., $y\in U_1$.

(c) Suppose $x=(x_0,x_1,x_2,x_3)\in U_4$. Then $\sum_{k=0}^3 x_k x_k^*=1$.
By comparing the constant terms, we conclude that
exactly one of the $x_k$ belongs to $U_1$ while the other vanish.

(d) Suppose $x=(x_k)\in U_4'$. Then $\sum_k x_k x_k^*=2-z^2-z^{-2}.$
Hence $\sum_k \CT(x_k x_k^*)=2$, which by the above remark implies that
exactly one of the $x_k$s is nonzero. This nonzero component
belongs to $(z-z^{-1})U_1$ by (b).
\end{proof}

Note that $\Ort_4(\bZ)=E_{16}\Sigma_4$ is the semidirect product of
the elementary abelian group $E_{16}$ of order 16 which acts on $\bZ^4$
by multiplying the coordinates with $\pm1$ and the symmetric
group $\Sigma_4$ of degree 4 which permutes the coordinates.

For $u=(u_0,u_1,u_2,u_3)\in U_1^4$ we define $\sigma_u\in\Ort(N)$ by
$\sigma_u(x)=(u_k x_k)$, where $x=(x_k)\in E$. Then
\begin{equation}
T=\{\sigma_u\,\,|\,\,u\in U^4_1\}
\end{equation}
is an abelian subgroup of $\Ort(N)$.

Any $\alpha\in \Ort_4(\bZ)$ extends uniquely to an $A$-linear element
of $\Ort(N)$, still denoted by  $\alpha$.
Thus we consider  $\Ort_4(\bZ)$ as a subgroup of $\Ort(N)$.
Define $\tau_0\in \Ort(N)$ by
$\tau_0(x_0,x_1,x_2,x_3)=(x_0^*,x_1,x_2,x_3)$, and
define similarly $\tau_k$ for $k=1,2,3$. Let  $\Gamma$ denote the  group of order $16$ generated by the $\tau_k$s.
The group $\Ort(N)$ admits the following factorization.

\begin{theorem} \label{Ort-2}
$\Ort(N)=T\Sigma_4\Gamma.$
\end{theorem}
\begin{proof}
Let $\vf\in \Ort(N)$. Since $\vf(U_4)=U_4$, for any index $i$ there
exist $u_i\in U_1$ and an index $i'$ such that $\vf(e_i)=u_i e_{i'}$.
Since the form $Q$ is $\vf$-invariant, the map $i\to i'$ must be a
permutation of $\{0,1,2,3\}$.

For any $a\in A$ we have $a+a^*\in A_0$ and so
$$ \vf((a+a^*)e_i)=(a+a^*)u_ie_{i'}. $$
Now let $a\in U_1$,  $a\ne\pm1$. Then each of the terms $\vf(ae_i)$,
$\vf(a^*e_i)$, $au_ie_{i'}$ and $a^*u_ie_{i'}$ belongs to $U_4$, and
the sum of the first two terms is not 0. It follows that
$\vf(ae_i)$ is equal to $au_ie_{i'}$ or $a^*u_ie_{i'}$.

Assume that there exist $a,b\in U_1$ different from $\pm1$ such
that $\vf(ae_i)$ $=au_i e_{i'}$ while $\vf(be_i)=b^*u_i e_{i'}$.
Then $\vf((a+b)e_i)=(a+b^*)u_ie_{i'}$, and by taking the norms we
obtain the contradiction $(a-a^*)(b-b^*)=0$. Consequently, one of
the following two identities: $\vf(ae_i)=au_ie_{i'}$ or
$\vf(ae_i)=a^*u_ie_{i'}$ must hold for all $a\in A$. Equivalently,
there exists $\ve_i\in\{0,1\}$ such that
$\vf\tau_i^{\ve_i}(ae_i)=au_ie_{i'}$ for all $a\in A$. Let
$$ u=(u_i)\in U_1^4 \quad \text{and} \quad
\tau=\prod_{i=0}^3 \tau_i^{\ve_i} \in \Gamma. $$
The composite $\beta=\sigma_u^{-1}\vf\tau$ is $A$-linear and we have
$\sigma_u^{-1}\vf\tau(e_i)=e_{i'}$ for all $i$.
Hence, $\vf=\sigma_u\beta\tau$ with $\sigma_u\in T$, $\beta\in\Sigma_4$
and $\tau\in\Gamma$.
\end{proof}

\section{Composition algebra structures on $E$} \label{proizvodi}

In this section we determine all $A_0$-bilinear multiplications ``$\cdot$''
on $E$ which satisfy the (polynomial) Lagrange identity
\begin{equation} \label{LagId}
N(x\cdot y)=N(x)N(y).
\end{equation}
Since $E$ is a free $A_0$-module of rank 8 and the form $Q$ is
nondegenerate, such algebra $(E,\cdot)$ will be a composition algebra
in the sense of the definition in \cite[p. 305]{MAK} provided that
it has an identity element.

We say that two $A_0$-bilinear multiplications $\star$ and $\diamond$ are
{\em equivalent} if there exist
$\sigma_1,\sigma_2,\tau\in \Ort(N)$ such that
$x\diamond y=\tau(\sigma_1(x)\star\sigma_2(y))$
for all $x,y\in E$.

The following theorem is our main result.

\begin{theorem} \label{glavna}
Any $A_0$-bilinear multiplication ``$\cdot$'' on $E$ satisfying the
Lagrange identity (\ref{LagId}) is equivalent to the multiplication
``$\circ$'' defined by Yang's formulae.
\end{theorem}

\begin{proof}
We shall reduce the proof to the special case where $e_0$ is the
identity element of the $A_0$-algebra $(E,\cdot)$. Then the assertion
of the theorem follows from the next theorem.

Define $A_0$-linear maps $L,R:E\to E$ by $L(u)=e_0\cdot u$ and
$R(u)=u\cdot e_0$. The Lagrange identity implies that $L$ and $R$
preserve $N$. By the remark \ref{primedba}, $L,R\in\Ort(N)$.

Next we use a well known argument due to, at least, Kaplansky
\cite{IK}. We have $N(R^{-1}(x)\cdot L^{-1}(y))=N(x)N(y)$ for
$x,y\in E$. With $x\star y=R^{-1}(x)\cdot L^{-1}(y)$, we have
\begin{eqnarray*}
&& (e_0\cdot e_0)\star x=R^{-1}(e_0\cdot e_0)\cdot L^{-1}(x)=
e_0 \cdot L^{-1}(x)=x, \\
&& x\star (e_0\cdot e_0) =R^{-1}(x)\cdot L^{-1}(e_0\cdot e_0)
= R^{-1}(x)\cdot e_0=x.
\end{eqnarray*}
Thus $e_0\cdot e_0$ is the identity element of the algebra
$(E, \star)$. As $e_0\cdot e_0\in U_4$, we have $e_0\cdot e_0=ae_i$
for some $a\in U_1$ and some index $i$. Hence, there exists
$\sigma\in\Ort(N)$ such that $\sigma(e_0)=e_0\cdot e_0$.
If $x\diamond y=\sigma^{-1}(\sigma(x)\star \sigma(y))$, then $e_0$
is the identity element of the algebra $(A, \diamond)$.
\end{proof}

We assume now that $e_0$ is the identity element of
$(E,\cdot)$. Consequently, this is a composition algebra
and the alternative laws
\begin{equation} \label{alt}
x\cdot(x\cdot y)=(x\cdot x)\cdot y, \quad
(x\cdot y)\cdot y=x\cdot(y\cdot y)
\end{equation}
are valid \cite[p. 306]{MAK}. The reader should consult this book
for additional properties of composition algebras. We state only
a few properties that we need.

Define the $A_0$-linear map {\em trace} $T: E\to A_0$ by 
$T(x)=Q(e_0,x)$. Then for any $x\in E$ we have 
\begin{equation} \label{kvadrat}
x \cdot x-T(x)x+N(x)e_0=0.
\end{equation}
Linearizing gives
\begin{equation} \label{trace}
x \cdot y+y \cdot x=T(x)y+T(y)x-Q(x,y)e_0,\,\forall\, x,y\in E.
\end{equation}
Clearly $Ae_0\perp A^3$ with respect to $Q$, where 
$A^3=Ae_1+Ae_2+Ae_3$.

Recall that we have extended the conjugation ${}^*$ from $A$ to $E$
in Section \ref{Octonion}. Moreover, note that
$$ (ae_0+v)^* =a^*e_0-v,\quad a\in A,\, v\in A^3. $$
It is easy to verify that
\begin{equation} \label{forma}
Q(x,z\cdot y^*)=Q(x\cdot y,z)=Q(y,x^*\cdot z),\,\forall x,y,z\in E.
\end{equation}

The fact that in the next theorem one can assert that the two
algebras are actually isomorphic is due to the referee.

\begin{theorem} \label{referee}
Every $A_0$-algebra $(E,\cdot)$ satisfying the Lagrange identity (\ref{LagId}) and having $e_0$ as the identity element is isomorphic to the Yang algebra $(E,\circ)$.
\end{theorem}

We shall break the proof into several lemmas.
\footnote{For a simpler proof see Section \ref{Dodatak}.}

\begin{lemma} \label{druga}
By replacing ``$\cdot$'' with an isomorphic multiplication, we
may also assume that
\begin{equation} \label{A-lin}
(ae_0)\cdot y=ay
\end{equation}
for all $a\in A$ and $y\in E$.
\end{lemma}

\begin{proof} Let us fix an index $i$ and $b \in U_1$. Then the
right multiplication by $be_i$ belongs to $O(N)$, and Theorem 3.4
yields an index $j$ as well as $u_j \in U_1$ such that $(ae_0)\cdot
(be_i) = u_jae_j$ for all $a \in A$ or $(ae_0)\cdot (be_i) =
u_ja^*e_j$ for all $a \in A$. Specializing $a$ to $1$ implies that
$j =i$ and $u_j = b$, hence $(ae_0)\cdot (be_i) =  abe_i$ for all 
$a \in A$ or $(ae_0)\cdot (be_i) =  a^*be_i$ for all $a \in A$.

Assume that $(ae_0)\cdot(b_1e_i)=ab_1e_i$ holds for all $a\in U_1$ and
$(ae_0)\cdot(b_2e_i)=a^*b_2e_i$ holds for all $a\in U_1$, for some
units $b_1,b_2$ different from $\pm1$. Then
$(ae_0) \cdot ((b_1+b_2)e_i) = (ab_1+a^*b_2)e_i$ holds for
all $a\in U_1$. By taking norms on both sides, we obtain that
$$ (1-a^2)b_1b_2^*+(1-(a^*)^2)b_1^*b_2=0 $$
for all $a\in U_1$. This is clearly a contradiction.

Since $U_1$ generates $A$ as an abelian group, we conclude that
either  $(ae_0)\cdot (be_i)=ab e_i$ for all $a,b\in A$ or
$(ae_0)\cdot (be_i)=a^*b e_i$ for all $a,b\in A$. In particular,
$(ae_0)\cdot(be_i)$ is $A$-linear in $b$.

If $(ae_0)\cdot (be_0)=a^*b e_0$ for all $a, b \in A$, then setting
$b = 1$ gives the contradiction: $ae_0 = a^*e_0$ for all $a \in A$.
Thus we have $(ae_0)\cdot (be_0)=ab e_0$ for all $a,b\in A$.
For $i\ne0$, there is an $\ve_i\in\{0,1\}$ such that
$\tau_i^{\ve_i}(\tau_i^{\ve_i}(ae_0)\cdot\tau_i^{\ve_i}(be_i))=abe_i$
for all $a,b\in A$.

If $\tau$ is the product of the $\tau_i^{\ve_i}$ with $i\ne0$,
then $\tau^2=1$ and the multiplication ``$\star$'' defined by $x\star
y=\tau(\tau(x)\cdot\tau(y))$ satisfies all the requirements.

Moreover, the map $\tau:(E,\star)\to (E,\cdot)$ is an isomorphism
of unital $A_0$-algebras.
\end{proof}

We assume from now on that the identity (\ref{A-lin}) holds.
This property will be shared by the modified multiplication which
will be introduced in Lemma \ref{peta}.

\begin{lemma} \label{treca}
For $a\in A$ and $x\in Ae_1+Ae_2+Ae_3$ we have $x\cdot(ae_0)=a^*x$.
\end{lemma}

\begin{proof} Recall that $Ae_1+Ae_2+Ae_3$ is the kernel of the 
trace $T$ and is perpendicular to $Ae_0$ relative to $Q$. By using 
(\ref{trace}), we obtain that 
\begin{eqnarray*}
x\cdot(ae_0) &=& x\cdot(ae_0)+(ae_0)\cdot x -(ae_0)\cdot x \\
 &=& T(ae_0)x+T(x)ae_0-Q(x,ae_0)e_0-ax \\
 &=& (a+a^*-a)x = a^*x.
\end{eqnarray*}
\end{proof}

\begin{lemma} \label{cetvrta} For $a,b\in A$ and $i\ne 0$ we have
$(ae_i)\cdot (be_i)=-ab^*e_0$.
\end{lemma}
\begin{proof} Since $N$ is anisotropic, it suffices to show that
$N((ae_i) \cdot  (be_i) + ab^*e_0) = 0$.
By expanding the LHS and applying (\ref{forma}) gives 
\begin{eqnarray*}
N((ae_i)\cdot(be_i)+ab^*e_0)&=&2N(a)N(b)+Q((ae_i)\cdot(be_i),ab^*e_0)\\
 &=& 2N(a)N(b) + Q(ae_i, (ab^*e_0) \cdot  (be_i)^*) \\
 &=& 2N(a)N(b) -Q(ae_i, (ab^*e_0) \cdot  (be_i)) \\
 &=& 2N(a)N(b) -N(b)Q(ae_i, ae_i) \\
 &=& 2N(a)N(b)- 2N(a)N(b) = 0.
\end{eqnarray*}
\end{proof}

\begin{lemma} \label{peta} By replacing ``$\cdot$'' with an isomorphic
multiplication, we may assume that
\begin{equation} \label{kvaternion}
e_i\cdot e_j=e_k=-e_j\cdot e_i
\end{equation}
for any cyclic permutation $(i,j,k)$ of $(1,2,3)$.
\end{lemma}
\begin{proof} Let $(i,j,k)$ be a cyclic permutation of $(1,2,3)$.
Since $(e_0+e_i)\cdot(e_0+e_j)=e_0+e_i+e_j+e_i\cdot e_j$ and
$N(e_0+e_j)=N(e_0+e_j)=2$, we have $N(e_0+e_i+e_j+e_i\cdot e_j)=4$.
As $e_i\cdot e_j\in U_4$, Lemma \ref{jednacina} implies that
$e_i\cdot e_j\in U_1 e_k$. A similar argument shows that
$e_j\cdot e_i\in U_1 e_k$.

In particular we have $e_1\cdot e_2=ue_3$ for some $u\in U_1$.
The multiplication ``$\star$'' defined by
$x \star y=\sigma_a^{-1}(\sigma_a(x)\cdot\sigma_a(y))$,
where $a=(1,1,1,u)$, has the previously stated property and sends
$(e_1,e_2)\to e_3$. By replacing ``$\cdot$'' with this new multiplication, we may assume that $e_1\cdot e_2=e_3$.

Since $(e_i+e_j)\cdot(e_i+e_j)=e_i\cdot e_j+e_j\cdot e_i-2e_0$,
we have $N(e_i\cdot e_j+e_j\cdot e_i-2e_0)=4$. Since
$e_i\cdot e_j+e_j\cdot e_i\in Ae_k$, it follows that this sum is 0,
i.e., $e_j\cdot e_i=-e_i\cdot e_j$. In particular, we have
$e_2\cdot e_1=-e_3$. Thus the assertion of the lemma is valid
if $(i,j,k)=(1,2,3)$.

The remaining equalities (\ref{kvaternion}) follow easily by using
the alternative laws (\ref{alt}). For instance, we have
$e_1\cdot e_3=e_1\cdot(e_1\cdot e_2)=(e_1\cdot e_1)\cdot e_2=
(-e_0)\cdot e_2=-e_2$.

To finish the proof, we point out that the map
$\sigma_a:(E,\star)\to (E,\cdot)$ is an isomorphism of unital
$A_0$-algebras.
\end{proof}

In view of the last lemma, we may assume now (and we do) that the
identities (\ref{kvaternion}) are valid for any cyclic
permutation $(i,j,k)$ of $(1,2,3)$.

\begin{lemma} \label{sesta}
For any cyclic permutation $(i,j,k)$ of $(1,2,3)$ we have
\begin{equation} \label{koef-konj}
(ae_i)\cdot(be_j)=a^*b^*e_k=-(be_j)\cdot(ae_i)
\end{equation}
for all $a,b\in A$. Consequently, the multiplications ``$\cdot$''
and ``$\circ$'' coincide.
\end{lemma}

\begin{proof} In view of (\ref{trace}), it suffices to prove the 
first equality. We first treat the case $a=1$. By applying the linearized left
alternative law, we find that
\begin{eqnarray*}
e_i \cdot (be_j) &=& e_i  \cdot ((be_0) \cdot e_j) \\
 &=&(e_i\cdot(be_0)+(be_0)\cdot e_i)\cdot 
	e_j-(be_0)\cdot(e_i\cdot e_j) \\
 &=& (b+b^*)e_k-be_k = b^*e_k.
\end{eqnarray*}
For arbitrary $a$, by using this special case, the linearized 
right alternative law and Lemma \ref{treca}, we obtain that
\begin{eqnarray*}
(ae_i)\cdot(be_j) &=& (e_i\cdot(a^*e_0))\cdot(be_j) \\
 &=& e_i\cdot((a^*e_0)\cdot(be_j)+(be_j)\cdot(a^*e_0))-(e_i\cdot (be_j)) \cdot (a^*e_0) \\
 &=& e_i\cdot((a + a^*)be_j)-(b^*e_k)\cdot(a^*e_0) \\
 &=& ((a + a^*)b^* -ab^*)e_k = a^*b^*e_k.
\end{eqnarray*}
\end{proof}

This concludes the proof of Theorem \ref{referee}, and also of
Theorem \ref{glavna}.

The referee supplied an alternative proof of Theorem \ref{referee}.
His proof proceeds first by changing scalars from the base ring
$\bZ[t]$ to its quotient field $\bQ(t)$, and then applying a theorem of
Thakur \cite{MT} (see also \cite[Chapter VIII, Exercise 6]{KMRT}) which
establishes a connection between octonion algebras and ternary
hermitian forms. In the case of Yang algebra, let us write
$E = Ae_0 \oplus A^3$, where $A^3$ is the column $A$-space.
Extending the conjugation of $A$ componentwise to $A^3$,
we have a hermitian form
$h: A^3 \times A^3 \to A$ given by $ h(v,w) = v^tw^*$.
With this notation, the Yang multiplication is given by
$$(ae_0 \oplus v)\circ (be_0 \oplus w) = (ab - h(v,w))e_0 \oplus
(aw + b^*v + v^*\times w^*)$$
where $a, b \in A$ and $v,w \in A^3$. The sign $\times$ stands for the ordinary cross product in $3$-space.

\section{Addendum} \label{Dodatak}

Alberto Elduque \cite{AE} has found a simpler proof of 
Theorem \ref{referee}. Here we present his simplification.
We point out that the involution ``$*$'' coincides with the
the involution ``${}^-$'' in \cite[p. 305]{MAK} and so it
is an involution of the algebra $(E,\cdot)$.

The first paragraph of the proof of Lemma \ref{druga} implies that
$Ae_0\cdot e_i=Ae_i$ for each $i$. By applying the involution ``$*$''
we obtain that also $e_i\cdot Ae_0=Ae_i$ for each $i$. By
(\ref{kvadrat}) we have $(ze_0)\cdot(ze_0)=-e_0+tze_0=z^2e_0.$
It follows that  $(xe_0)\cdot(ye_0)=xye_0$ for all $x,y\in A$.
By (\ref{kvadrat}), we also have $e_1\cdot e_1=-e_0$.
By applying \cite[Lemma (7.1.7)]{MAK}, with $B=Ae_0$ and $v=e_1$,
we conclude that $P=Ae_0+Ae_1$ is a subalgebra of $(E,\cdot)$.
From the proof of that lemma it also follows that the multiplications
``$\cdot$'' and ``$\circ$'' coincide on $P$.

From (\ref{trace}) we deduce that 
$e_2\cdot P\subseteq P^\perp=Ae_2+Ae_3$.
Since $x\to e_2\cdot x$ is an orthogonal transformation on $E$,
we have $Ae_2=e_2\cdot Ae_0 \perp e_2\cdot Ae_1$. It follows 
that $e_2\cdot Ae_1\subseteq Ae_3$. Similarly, 
$e_2\cdot Ae_3 \subseteq Ae_1$. Thus we must have 
$e_2\cdot Ae_1=Ae_3$ and $e_2\cdot Ae_3=Ae_1$, and so
$E=P\oplus e_2\cdot P$.
We now apply \cite[Lemma (7.1.7)]{MAK} for the second time by taking
$B=P$ and $v=e_2$. We conclude that the multiplications
``$\cdot$'' and ``$\circ$'' coincide on $E$.

Note that this proof does not use any of the lemmas 4.4-7, and 
uses only the first paragraph of the proof of Lemma \ref{druga}.

\end{document}